\documentclass[psamsfonts]{amsart}
\usepackage[utf8]{inputenc}
\usepackage{amsfonts}
\usepackage{hyperref}
\hypersetup{
pdftitle={Recurrent Operators on Function Spaces},
pdfsubject={Mathematics, },
pdfauthor={Gabriela},
pdfkeywords={}
}
\usepackage{amsmath}
\usepackage{xcolor}
\usepackage{amsthm}
\usepackage{pdflscape}
\usepackage{pgfplots}
\usepackage{mathrsfs}

\newtheorem{thm}{Theorem}[section]
\newtheorem{cor}[thm]{Corollary}
\newtheorem{prop}[thm]{Proposition}
\newtheorem{lem}[thm]{Lemma}

\newtheorem{quest}[thm]{Question}

\theoremstyle{definition}
\newtheorem{defn}[thm]{Definition}

\newtheorem{exmp}[thm]{Example}

\theoremstyle{remark}
\newtheorem{rem}[thm]{Remark}

\makeatletter
\let\c@equation\c@thm
\makeatother
\numberwithin{equation}{section}

\title{Recurrent Operators on Function Spaces}
\begin{document}
\author{Gabriela Bulancea and H\'ector N. Salas}

\title{Recurrent Operators on Function Spaces}

\address{Department of Mathematical Sciences\\George Mason University\\4400 University Dr\\Fairfax, VA
22030} \email{gbulance@gmu.edu}
\address{Department of Mathematical Sciences\\University of Puerto Rico, Mayag\"{u}ez, PR 00681-9018} \email{hector.salas@upr.edu}
\thanks{}
\subjclass[2000]{Primary 37B20, 47B38. Secondary: 30D05, 47A16, {46J10, 15A18}} 
\keywords{Recurrent operator.  Iteration of composition operators with linear fractional symbols. Hypercyclic operators. Rigid and uniformly rigid operators.}
\date{}

\begin{abstract}
Although recurrence for dynamical systems has been studied since the end of the nineteenth century, the study of recurrence for linear operators started with papers by Costakis, Manoussos and Parissis in 2012 and 2014. We explore recurrence in Banach algebras,
in the space of continuous linear operators on $\omega=\mathbb{C}^{\mathbb{N}},$ and for composition operators whose symbols are linear fractional transformations, acting on weighted Dirichlet spaces. In particular, we show that  composition operators with a parabolic non automorphism symbol are never recurrent. Our results relate to work by Gallardo-Guti\'errez and Montes-Rodr\'iguez in their 2004 AMS Memoir in which they show, among other things, that composition operators induced by parabolic non automorphisms are never hypercyclic in weighted Dirichlet spaces. Our work also relates to a more recent paper by Karim, Benchiheb, and Amouch.

\end{abstract}
\maketitle
\tableofcontents
\section {Introduction}

Recurrence is a very fruitful notion in dynamical systems. Poincar\'{e}'s recurrence theorem
in the late nineteenth century was the initial landmark. Let $f:X\longrightarrow X$ and for each $x\in X$ consider the orbit $Orb(f,x)=\{f^n(x):n=0,1,2,3,\cdots\}.$ The interesting case is when $X$ has some kind of structure: measure space, topological space, etc and $f$ has the adequate property as volume preserving (measure preserving), continuous, etc. A vector $x$ is {\it recurrent} if it visits the appropriate subsets containing $x$ infinitely often. The function $f$ is {\it recurrent} if ``almost" any vector is recurrent. 
The study of recurrence in the linear setting was initiated by Costakis, Manoussos and Parissis in \cite{Costakis} and Costakis and Parissis in \cite{Parissis}. Another paper we considered while working on this topic is the paper of  Karim, Benchiheb, and Amouch \cite{Karim}. Our subject was expounded by Grosse-Erdmann in his talk {\it Recurrent Operators} presented in the online conference Four Days in Linear Dynamics organized at the University of Campania Luigi Vanvitelli in June 2021.\\
\par
Here are the precise definitions of recurrence and two stronger forms of recurrence, rigidity and uniform rigidity of operators acting on a Banach space $X$, as they appear in \cite{Costakis}. 

\begin{defn} An operator $T$ acting on $X$ is called recurrent if for every open set
$U\subset{X}$ there exists some $k\in{\mathbb{N}}$ such that
$$U \cap{T ^{-k} (U)}\ne{\emptyset}.$$
A vector $x\in{X}$ is called recurrent for $T$ if there exists a strictly increasing sequence
of positive integers $(k_n)_{n\in{\mathbb{N}}}$ such that
$T^{k_n}x\rightarrow{x}$,
as $n\rightarrow{\infty}$. We will denote by Rec$(T)$ the set of recurrent vectors for $T$.
\end{defn}
\begin{rem}
    The space $X$ in the definition above could also be a Fr\'{e}chet space.
\end{rem}
\begin{defn}
    An operator $T$ acting on $X$ is called rigid if there exists an increasing
sequence of positive integers $(k_n)_{n\in{\mathbb{N}}}$ such that
$T^{k_n} x\rightarrow{x}$ for every $x\in{X}$,
i.e., $T^{k_n}\rightarrow{I}$ in the strong operator topology (SOT).
\\
An operator $T$ acting on $X$ is called uniformly rigid if there exists an increasing
sequence of positive integers $(k_n)_{n\in{\mathbb{N}}}$ that goes to $\infty$ such that
	$$||T^{k_n}-I||=\sup_{||x||\le{1}}{||T^{k_n}x-x||}\rightarrow{0}.$$
\end{defn}

In the following, we will be using the characterization of recurrent operators in Proposition 2.1 in \cite{Costakis} that states that the operator $T$ on the Banach space $X$ is recurrent if and only if $\overline{\text{Rec}(T)}=X$.

Trivial examples of recurrent operators are $e^{i2\pi c} I$
(a multiply of the identity); when $c\in \mathbb Q$ the operator is periodic, whereas when $c \in \mathbb R \setminus \mathbb Q$, it is not periodic, but Kroenecker's theorem does the job to show that it is still recurrent. Direct sums of these operators are also recurrent.

A  related but more restrictive notion than recurrence is  hypercyclicity, which is a very active field of research, see the book by Bayart and Matheron \cite{Bayart}. A continuous linear operator $T:X\longrightarrow{X}$,
where $X$ is a Banach space or, more generally, a topological vector space, is hypercyclic if each point in a dense subset of $X$ has a dense orbit. Thus, if $T$ is hypercyclic on $X$ and $x\in{X}$ is such that $Orb(T,x)$ is dense in $X$, then for any neighborhood $V$ of $x$ a subsequence of the orbit is contained in $V$, which shows that $x$ is recurrent for $T$. Consequently, if $T$ is hypercyclic, then $T$ is also recurrent. In the traditional, non linear case, a recurrent point has a ``bounded" orbit; however, if $x$ is a hypercyclic vector, its orbit is not bounded. 

We now give an overview of the paper.

In the second section, we study recurrence of multiplication operators in the context of commutative Banach algebras. We show that the inverse of a uniformly rigid operator is also uniformly rigid. (See Question 2.20 in \cite{Costakis}.) In particular, we look at multiplication operators on the disk algebra $\mathcal{A}$ and on the algebra $H^\infty$.

In the third section, we characterize those recurrent operators on $\omega$ whose representing matrices are lower triangular. Recall that  the space $\omega=K^{\mathbb{N}}$ is the countably infinite product of $K$ endowed with the product topology, which is a Fr\'{e}chet space. The field $K$  could be either $\mathbb{C}$ or $\mathbb{R},$ but we work only with the complex numbers. 

The fourth section is an interlude in which we provide an example of a non-recurrent operator related to example 7.9 in \cite{Costakis}, but in which the Cantor set used has positive measure.

In the last section, we revisit the recurrence of composition operators with linear fractional transformation symbols (LFT) acting on weighted Dirichlet spaces $S_{\nu}$.
In \cite{Costakis}, the recurrence of these composition operators was studied in the Hardy space context, while their hypercyclic behavior on the spaces $S_\nu$ was studied in the monograph by Gallardo-Guti\'errez and Montes-Rodr\'iguez in \cite{Gallardo-Montes}. In particular, they show that when the symbol is a parabolic non-automorphism the operator is not hypercyclic (Chapter 4 of \cite{Gallardo-Montes}). Recently, Karim,
Benchiheb, and Amouch \cite{Karim}  provided a characterization of the recurrent composition operators on $S_\nu$. While the statements of their results are true, their proofs for two of the cases are incorrect. We provide the correct proofs for these cases. In particular, we prove that the operators induced by parabolic non-automorphisms are not recurrent, so, they are not hypercyclic. The proof, which is short and direct, is obtained by using the tools in the "An appetizer" section on page 52 in \cite{Gallardo-Montes}. We conclude with a discussion of the recurrence (or lack thereof) of composition operators on $\mathcal{A}$ and on the algebra $H^\infty$.
\\
\par
{\bf Notation:} We will be using the standard notation denoting the set of natural numbers by $\mathbb{N}$, the unit circle by $\mathbb{T}$, the open unit disk by $\mathbb{D}$, and the closed unit disk by $\overline{\mathbb{D}}$. The field of scalars in all our spaces is $\mathbb C.$

\section{Banach Algebras}
Let $A$ be a (commutative) Banach algebra with identity, which will be denoted by $1$. A good reference for commutative Banach algebras is the book by Stout \cite{Stout}. 
\par
An element $a\in{A}$ induces a multiplication operator on $A$, $M_a:A\rightarrow{A}$, $M_a(x)=ax$. Then $||M_a||=||a||$ and the spectrum of the operator $M_a$ coincides with the spectrum of $a$ as an element of the Banach algebra $A$. In particular, $M_a$ is invertible if and only if $a$ is invertible in $A$; when $a$ is invertible $(M_a)^{-1}=M_{a^{-1}}$. In the following, we are interested in studying the connection between the properties of $a$ as an element of the Banach algebra and the properties of the operator $M_a$ of being recurrent, rigid, or uniformly rigid. Notice that for each $a$ the zero vector is recurrent for $M_a$, while $M_0$ is never recurrent.
\par 
Since there are hypercyclic operators that are not invertible, there are recurrent operators that are not invertible. A nice example of such an operator is $\lambda{B}$, where $B$ is the backward shift on the Hardy space $H^2$ and $|\lambda|>1$. This is Rolewicz's Theorem, page 110 in Shapiro's book \cite{Shapiro}.
In \cite{Costakis}, the authors proved that if a recurrent operator is invertible, then its inverse is also recurrent. While uniformly rigid operators are invertible, it is an open question whether a rigid operator has to be invertible, which is part of question 2.20 in \cite{Costakis}. 
We can use the Newmann series to show that the inverse of a uniformly rigid operator is also uniformly rigid, answering the other part of the question 2.20 in \cite{Costakis}.

\begin{prop}
    Let $T$ be a uniformly rigid operator on the Banach space $X$. Then $T^{-1}$ is also uniformly rigid.
\end{prop}
\begin{proof} Since $T$ is uniformly rigid, there is a sequence $(n_k)_{k=1}^\infty$ in $\mathbb{N}$, such that $||T^{n_k}-I||\rightarrow{0}$, as $k$ goes to infinity. Thus, there exists a positive integer $k_0\in{\mathbb{N}}$ such that $||T^{n_k}-I||<\frac{1}{2}$ for all $k>k_0$. For $k>k_0$, we have that $$(T^{n_k})^{-1}=\sum_{l=0}^\infty{(I-T^{n_k})^l}.$$
Thus, 
$$I-(T^{n_k})^{-1}=I-\sum_{l=0}^\infty{(I-T^{n_k})^l}=\sum_{l=1}^\infty{(I-T^{n_k})^l}=(I-T^{n_k})\sum_{l=0}^\infty{(I-T^{n_k})^l}.$$
From this we have that
$$||I-(T^{n_k})^{-1}||\le{||I-T^{n_k}||}\sum_{l=0}^\infty{||I-T^{n_k}||^l}=$$ $$||I-T^{n_k}||\frac{1}{1-||I-T^{n_k}||}\le{2||I-T^{n_k}||}, $$
which shows that $||I-(T^{n_k})^{-1}||\rightarrow{0}$, as $k\rightarrow{\infty}.$
Therefore, $T^{-1}$ is uniformly rigid.

\end{proof}
\begin{prop}
    If the multiplication operator $M_a$ is recurrent, then $M_a$ is invertible. In this case $M_a^{-1}$ is also recurrent.
\end{prop}

\begin{proof}
 Since $M_a$ is recurrrent on $A$, the set of recurrent vectors is dense in $A$. Thus, there exists a vector $b$ in $B(1, \frac{1}{4})$, such that $b$ is recurrent. Since $||b-1||<1$, $b$ is invertible in $A$.
 Given that $b$ is recurrent, there exists and increasing sequence of natural numbers $(n_k)_{k=1}^\infty$ such that $||M_a^{n_k}b-b||$ converges to $0$ as $k$ goes to infinity. Hence, there exists $k_0>1$ such that $||M_a^{{n_k}_0}b-b||<\frac{1}{4}$. Then we have that
   $$||a^{{n_k}_0}b-1||\le{||a^{{n_k}_0}b-b||+||b-1||}<\frac{1}{2}<1,$$
   which shows that $a^{{n_k}_0}b$ is invertible, and so, $a^{n_k}$ is invertible. It follows that $a$, and so $M_a$, is invertible and recurrent by Proposition 2.6 in \cite{Costakis}.
\end{proof}
\begin{prop}
    The following are equivalent.
    \begin{enumerate}
        \item $M_a$ is recurrent on $A$.
        \item $M_a$ is rigid on $A$.
        \item $M_a$ is uniformly rigid on $A$.
    \end{enumerate}
\end{prop}
\begin{proof}
    It is enough to prove that the first statement implies the third one.
    Let $M_a$ be recurrent on $A.$ By the above result, it follows that $a$ is invertible.
    Since the set of invertible elements is open in $A$ and the set of recurrent vectors is dense in $A$, we have that there is a recurrent vector $b$ which is invertible. 
    Given that $b$ is recurrent, there exists an increasing sequence of natural numbers $(n_k)_{k=1}^\infty$ such that $||M_a^{n_k}b-b||$ converges to $0$, as $k$ goes to infinity. Thus, since 
    $$||1-a^{n_k}||=||(1-a^{n_k})bb^{-1}||\le{||(1-a^{n_k})b||||b^{-1}||},$$
    we have that $||1-a^{n_k}||\rightarrow{0}$, as $k$ goes to infinity.
    But 
    $$||M_a^{n_k}-I||=\sup_{||x||=1}{||M_a^{n_k}x-x||}=\sup_{||x||=1}{||a^{n_k}x-x||}\le{}$$  $$\sup_{||x||=1}{||a^{n_k}-1||||x||}\le{||a^{n_k}-1||},$$
    which shows that $||M_a^{n_k}-I||\rightarrow{0}$, as $k$ goes to infinity, which means that $M_a$ is uniformly rigid.
\end{proof}
\begin{rem}
    Parts i, ii, and iii of Theorem 7.1 in \cite{Costakis} are a particular case of this result.
\end{rem}
Next, we are interested in studying the conditions under which having one recurrent vector is sufficient for the operator $M_a$ to be recurrent.
\begin{prop}
  If $b$ is a recurrent vector for $M_a$, then every vector in the principal ideal generated by $b$ is a recurrent vector for $M_a$.  
\end{prop}

\begin{proof}
    Since $b$ is a recurrent vector for $M_a$, there exists an increasing sequence of natural numbers $(n_k)_{k=1}^\infty$ such that $||M_a^{n_k}b-b||=||a^{n_k}b-b||$ converges to $0$ as $k$ goes to infinity. Let $c$ be an element of $A$, then we have that
    $$||M_a^{n_k}(bc)-bc||=||a^{n_k}(bc)-bc||=||(a^{n_k}b-b)c||\le{||a^{n_k}b-b||||c||},$$
    from which it follows that $||M_a^{n_k}(bc)-bc||$ converges to $0$ as $k$ goes to infinity.
    Thus, $bc$ is recurrent for $M_a$.
\end{proof}
\begin{cor}
   If $M_a$ has an invertible recurrent vector, then $M_a$ is recurrent. 
\end{cor}
\begin{proof}
    Let $b$ be an invertible recurrent vector. Then every vector in the principal ideal generated by $b$ is recurrent for $M_a$. Since $b$ is invertible, this principal ideal is $A$, thus, every vector in $A$ is recurrent for $M_a$, which makes $M_a$ recurrent on $A$.
\end{proof}

In the following, we will consider the image $\hat{A}$ of the Banach algebra $A$ under the Gelfand transform. In this context, the Gelfand transform of an element $b$ of $A$ will be denoted by $\hat{b}$. It is known that if $A$ is a commutative semisimple Banach algebra with unity, then the Gelfand transform is a one-to-one, norm decreasing transformation of $A$ onto $\hat{A}$, which is a function algebra on $X$ (the maximal ideal space of $A$) that contains the constants and separates the points of $X$. Every element $\phi\in{X}$ is a complex homomorphism of $A$ and we have that $\hat{a}(\phi)=\phi(a)$ and that $\hat{a}(\phi)=0$, means that $a$ is in the maximal ideal that is the kernel of $\phi$. 
\par
For the disk algebra, $\mathcal{A}$ = the algebra of functions that are analytic on $\mathbb{D}$ and continuous on $\overline{\mathbb{D}}$, the maximal ideal space is the closed unit disk. We also know that for $H^\infty=$ the algebra of bounded analytic functions on $\mathbb{D}$, the open unit disk is dense in the maximal ideal space.
\begin{prop} 
  If there exist elements $b_1, b_2, \ldots, b_p$  that are recurrent for $M_a$ with the same sequence $(n_k)_{k=1}^\infty$ such that $\sum_{i=1}^p|\hat{b}_i(\phi)|\ge{\delta}$ for $\phi$ in a dense subset of $X$ for some positive $\delta$, then $M_a$ is recurrent.
\end{prop}
\begin{proof}
    The condition $\sum_{i=1}^p|\hat{b}_i|\ge{\delta}$ implies that the $\hat{}{b_i}$'s have no common zero, so that there is no maximal ideal that contains all of them. Thus the ideal generated by the the  elements $b_i$ cannot be contained in any (proper) maximal ideal so that it is the entire algebra $A$. Hence, there exist elements $c_1, c_2, \ldots. c_p$ of $A$ such that $\sum_{i=1}^p{b_ic_i}=1$, from which $(1-a^{n_k})\sum_{i=1}^pb_ic_i=1-a^{n_k}$.
    Therefore,
    $$||1-a^{n_k}||=||\sum_{i=1}^p(1-a^{n_k}){b_ic_i}||\le{\sum_{i=1}^p||(1-a^{n_k})b_i||||c_i||}\le{M\sum_{i=1}^p||(1-a^{n_k})b_i||},$$
    where $M=\max_{i=\overline{1, p}}{||c_i||}$.
    Since $||(1-a^{n_k})b_i||$ converges to $0$ for every $i$, $1\le{i}\le{p}$, we have that $||1-a^{n_k}||$ converges to $0$, which shows that $1$ is a recurrent vector and this implies by the previous corollary that $M_a$ is recurrent.
\end{proof}
\par
In the following, we will denote the zero set of an element $\hat{b}\in{\hat{A}}$ by $Z(\hat{b})$, $$Z(\hat{b})=\{\phi\in{X}: \hat{b}(\phi)=0\}.$$
\begin{prop}
    Let $b$ be a recurrent vector for $M_a$ such that $||M_a^{n_k}b-b||=||a^{n_k}b-b||$ converges to $0$ as $k$ goes to infinity, where $(n_k)_{k=1}^\infty$ is an increasing sequence of natural numbers. Then $\hat{a}^{n_k}(\phi)\rightarrow{1}$ for every $\phi$ in the complement of $Z(\hat{b})$ and, so, $|\hat{a}(\phi)|=1$ for $\phi\in{X\setminus{Z(\hat{b})}}$.
\end{prop}
\begin{proof}
    Let $\phi\in{X\setminus{Z(\hat{b})}}$. Since $a^{n_k}b-b\rightarrow{0}$ in $A$ and $\phi$ is continuous on $A$, we have that $\phi(a^{n_k}b-b)=\phi(a^{n_k}-1)\phi(b)\rightarrow{0}$. Given that $\phi(b)\ne{0}$, we have that $\phi(a^{n_k}-1)=\phi(a^{n_k})-\phi(1)\rightarrow{0}$, from which we have that $\phi(a^{n_k})=(\phi(a))^{n_k}\rightarrow{1}.$ It follows that $|\phi(a)|=1$.
    Therefore, the complement of the set $U=\{\phi\in{X}:|\phi(a)|=1\}$ is contained in the zero set of $\hat{b}$, $Z(\hat{b})$.
\end{proof}
\begin{cor}
    Under the hypotheses of the proposition above, if the interior of the zero set of $\hat{b}$ is empty, then $|\hat{a}|=1$ on $X$.
\end{cor}
\begin{proof} 
    The complement of the set $U$ on which $\hat{a}$ takes unimodular values is open and contained in $Z(\hat{b})$, and since $Z(\hat{b})$ doesn't contain any open set, we have that it has to be the empty set. Therefore, $|\hat{a}|=1$ on $X$. In particular, this implies that $a$ is invertible in $A$.
\end{proof}
When we apply this result to the disk algebra and to the algebra $H^\infty$, we obtain the following.
\begin{cor} Let $A$ be the disk algebra $\mathcal{A}$ or the algebra $H^\infty$.
If a multiplication operator $M_a$ on $A$ has a non-zero recurrent vector $b$, then $M_a$ is recurrent.
\end{cor}
\begin{proof}
According to the previous proposition, $|\hat{a}(\phi)|=1$ for all $\phi\in{X\setminus{Z(\hat{b})}}$. Recall that $\mathbb{D}$ is dense in the maximal ideal space of $A$.
The complement of the zero set of $\hat{b}$ contains an open disk $B\subset{\mathbb{D}}$. This implies that $|\hat{a}|=1$ on $B$ and, since $\hat{a}$ is analytic on $B$, we have that $\hat{a}=c$ on $B$. Thus, the function $\hat{a}-c$, which is analytic on the open unit disk $\mathbb{D}$, vanishes on the nonempty open subset $B$ of $\mathbb{D}$, therefore it is constantly $0$ on $\mathbb{D}$. Hence, $\hat{a}=c$, where $|c|=1$. But multiplication operators induced by a unimodular constant are recurrent, so the operator $M_a$ is recurrent.
 \end{proof}

\section{The Space $L(\omega)$} 

It is well known that the space $\omega=K^{\mathbb{N}}$, where $K=\mathbb{R}$ or $\mathbb{C}$,  is a Fr\'{e}chet space. A standard way to see that is the following. Let $x=(x_1, x_2, \cdots) \in \omega$ and  $p_n(x)=\sqrt{\sum_{i=1}^n|x_i|^2}.$ Then the topology in $\omega$ is given by the metric 
$d(x,y)=\sum_{n=1}^{\infty}2^{-n}\frac{p_n(x-y)}{1+p_n(x-y)}.$ For us, $K$ will be $\mathbb{C}$.

The space of operators (linear and continuous) from $\omega$ into itself will be denoted, as usual, by $L(\omega)$. An important consequence of $\omega$ being metrizable is that the open mapping theorem holds. We will identify the  operators with their representing matrices with respect to the canonical basis $\{e_k : k = 1, 2, \ldots\}$ where the $k$th coordinate of $e_k$ is $1$ and all the others are zero. 
An element $x$ of $\omega$ can be written as $\sum_{k=1}^\infty{x_ke_k}$, with $x_k\in{K}$, or as $(x_k)_{k=1}^\infty$. The set $\{k:x_k\ne{0}\}$ is called the support of $x$ and it is denoted by $supp(x)$. We will denote by $D$ the dense set of the finite span of $\{e_k:k = 1, 2, \ldots\}$, in other words, $D$  is the set of all vectors whose support is a finite set. 
In what follows, $A = (a_{i, j })$ with $i, j = 1, 2, 3, \ldots$ will be the representing matrix of an element of $L(\omega)$. Here are some of the results in \cite{Hector_omega} that describe the hypercyclic behavior of some classes of elements in $L(\omega)$. 
\begin{prop} [Proposition 1 \cite{Hector_omega}]
Let $A$ be an infinite matrix. Then $A\in{L(\omega)}$ if and only if every row has only
a finite number of entries different from zero.
\end{prop}

\begin{cor}If $A\in{L(\omega)}$ is one-to-one and onto, then $A^{-1} \in L(\omega )$ and so every row of $A^{-1}$ has only
a finite number of entries different from zero.
\end{cor}

\begin{proof} Apply the Open Mapping Theorem for the continuity of $A^{-1}$and then the proposition above to reach the conclusion about the entries. \end{proof}
As usual, a matrix $A$ is said to be a diagonal matrix if all the entries not on the main diagonal are $0$. 
\begin{defn}
The matrix $A\in{L(\omega)}$ is said to be upper staircase if $(n_i)_{i=1}^{\infty}$
is strictly increasing and $i<n_i$ for all $i$, where $n_i=\max{\{n: a_{i,n}\ne{0}\}}$. The set of upper staircase operators will be denoted by $S$.  
\end{defn}
\begin{defn}
   Let $\mathcal{A}$ be a class. Then $\{BAB^{-1}:A\in{\mathcal{A}, \hspace{.1in}B\in{L(\omega)}}\}$ will be denoted by $\mathcal{A}^s$. 
\end{defn} 
\begin{thm}[Theorem 2 \cite{Hector_omega}]If $A\in{S^s}$ , then $A$ is hypercyclic. Furthermore, if $K=\mathbb{C}$, then $A$ is chaotic.
\end{thm}
\begin{prop}[Proposition 6 \cite{Hector_omega}] Let $A=(a_{i,j})\in{L(\omega)}$ be lower triangular. Then $A$ is not hypercyclic. If the main diagonal of $A$ is zero and $f(\lambda)=\sum_{i=1}^{\infty}{f_n\lambda^n}$ is a formal analytic function, then $f(A)\in{L(\omega)}$.
\end{prop}
\par Conejero and B\`{e}s \cite{Bes-Conejero} studied denumerable sets of non-zero polynomials in a fixed backward weighted shift and proved that they have a common hypercyclic vector (actually an infinite-dimensional closed subspace of hypercyclic vectors, except the zero vector). The representing matrices of these operators are upper triangular.

From Theorem 4.1 in \cite{Costakis} we know that a lower triangular operator on $\mathbb{C}^n$ is recurrent if and only if it is similar to a diagonal matrix with entries in $\mathbb{T}$. In preparation for the main result of this section, which presents a  characterization of the recurrent lower triangular operators on $\omega$, we will prove a finite dimensional result. 

\begin{exmp} First, we show directly that a lower triangular and non-diagonal matrix $A$ with two different unimodular diagonal entries is recurrent (compare this with the corresponding part of the proof of Theorem 4.1 in \cite{Costakis}).    \[A = \begin{pmatrix} 1 & 0 \\ a & \lambda \end{pmatrix}  \Longrightarrow A^n = \begin{pmatrix} 1 & 0 \\ a\frac{\lambda^n-1}{\lambda-1} & \lambda^n \end{pmatrix}\] 
If $lim_{j\to \infty} \lambda^{n_j}=1,$ then $A^{n_j}$ goes to the identity matrix. 
\end{exmp}

\begin{rem} A recurrent operator  $A=(a_{i,j})_{n\times n}$ on the finite dimensional normed space $\mathbb{C}^n$ is uniformly rigid. 
\end{rem}
Since all norms on $\mathbb{C}^n$ are equivalent, it is convenient to 
see $\mathbb{C}^n$ as a Hilbert space 
with $\{e_j:j\leq n\}$ the canonical orthonormal basis. We identify the operator $A$ with its representing matrix with respect to the canonical basis and we have that $A(e_j)=\sum_{i=1}^n a_{i,j}e_i.$ 
\begin{lem} \label{finite}
 Let $A=(a_{i,j})_{n\times n}$ be a recurrent lower triangular matrix on $C^n.$ Suppose $i_1<\cdots < i_k\leq n$ and 
$\lambda=a_{i_1,i_1}=\cdots=a_{i_k,i_k}$ and $ a_{j,j}\neq \lambda$ if $ j\notin\{i_1,\cdots,i_k\}.$ Then there are $k$ eigenvectors for $\lambda$ of the form \[v_{i_p}=e_{i_p}+\sum_{j=i_p+1}^n x_j(i_p)e_j\] where $x_j(i_p)=0$ whenever $j=i_s$ and $p<s\leq{k}.$
\end{lem}

\begin{proof}
We first construct $v_{i_1}$, the other vectors can be constructed in a similar fashion. The construction of $v_{i_1}$ will take $k$ steps. To avoid so many indexes we will call its coeficients just $x_j.$

Step 1. Set $x_j=0$ for $j<i_1$ and $x_{i_1}=1$  For $i_{1}<v<i_{_2}$ we can solve  the equation \[\sum_{u=1}^{v-1} a_{v,u}x_u+a_{v,v}x_v=\lambda x_v\] for $x_v$, since $a_{v, v}\ne{\lambda}.$
\\
For $v=i_2,$ let $\sum_{u=1}^{v-1} a_{v,u}x_u=w; $ We claim that $w=0.$
Let $f=e_{i_1}+\sum_{j=i_1+1}^{i_2-1}x_je_j.$ Then
\[A(f)=\lambda f+we_{i_2} +g_1\]
where $g_1\in E=span \{e_j:i_2<j\leq n\}$ and $E$ is invariant under $A$. We also have that
\[ A(e_{i_2})=\lambda e _{i_2}+h\] with  $h\in E=span \{e_j:i_2<j\leq n\}.$ We can use induction on $m$  there is $g_m\in E=span \{e_j:i_2<j\leq n\}$ such that
\[A^mf=\lambda^mf+m\lambda^{m-1}we_{i_2}+g_m.\]
Here is the inductive step.
Assume that
    \[A^{m-1}(f)=\lambda^{m-1}f +(m-1)\lambda^{m-2}we_{i_2}+g_{m-1}\]
 Then
 \[AA^{m-1}(f)=A(\lambda^{m-1}f +(m-1)\lambda^{m-2}we_{i_2}+g_{m-1})\]
 The first two  terms on the right are:
  \[A(\lambda^{m-1}f)=\lambda^{m-1}(\lambda f+we_{i_2}+g_1)\]
 \[A((m-1)\lambda^{m-2}we_{i_2})=(m-1)\lambda^{m-2}w(\lambda e_{i_2}+h)\]
 Collecting the terms with $f$ and $e_{i_2}$ we have
\[\lambda^mf+m\lambda^{m-1}we_{i_2}\]
 the remainig terms are
\[\lambda^{m-1}g_1+(m-1)\lambda^{m-2}wh+A(g_{m-1})=g_m \in E,\] which proves that the equality holds for all $m$.
\\
From this, we have that for every $m\in{\mathbb{N}}$,
\[0\leq |\langle A^mf,e_{i_2}\rangle |=m|w|\leq ||A^mf-f||\]
since $|\lambda|=1.$ But since $A$ is recurrent we find a subsequence for which  the right limit is 0. Thus $w=0.$

Step 2.  If $k=2$ and $i_2=n$ we are done; if $k=2$ but $i_2\ne{n}$, we can get the remaining $x_j$ and therefore $v_{i_1}$ is obtained.
\\
If $k>2$, we find the values $x_j$ for $i_2<j<i_3$ and find that
for $v=i_3$ again we have that $\sum_{u=1}^{v-1} {a_{v,u}x_u}=0.$  
\\
We continue the above  process and after $k$ steps the required eigenvector $v_{i_1}$ is obtained. 
\\
The same procedure works well even if  two of the indexes $i_s$ are consecutive.
\\
The remaining eigenvectors for $\lambda$ are similarly obtained.
\end{proof}

\begin{rem}
In the finite dimensional case, a similar result is obtained for recurrent upper triangular matrices.
This is a simple consequence of the fact that we can relabel the orthonormal basis by reversing the order, i.e, $f_1=e_n, ~f_2=e_{n-1}\cdots f_n=e_1.$ In $L(\omega)$, that can not be done. Moreover, there are staircase operators that are upper triangular and have all their diagonal entries equal to $0$ which are recurrent since they are hypercyclic. One example of such an operator is the backward weighted shift.

\end{rem}
For the proof of the main result of this section, we will need the following lemma, which we will apply to the finite dimensional square block in the upper left corner of the infinite matrix whose recurrence we will study.

The lemma will also be used in the proof of Theorem 5.12. Another point worth noting is that the spaces involved need not be separable.
\begin{lem} \label{partition}
Let $E=F\oplus G$ be a Frechet space with metric $d$ which satisfies that there is $\delta >0$ with $\delta\big(d(f,0)+d(g,0)\big)\leq d(f+g,0)$ for all $f\in F,g \in G.$ Let $H\in L(F),~C\in L(G),$ and $B:F \longrightarrow G$ linear and continuous. If $f+g$, where $f\in{F}$ and $g\in{G}$, is recurrent for
\[T=\begin{pmatrix}
H & 0\\
B & C
\end{pmatrix},\]
 then $f$ is recurrent for $H.$  
 In addition, if $T$ is recurrent operator, so is $H.$
\end{lem}

\begin{proof}
First we calculate
\[T^n=\begin{pmatrix}
H^n & 0\\
\sum_{j=0}^{n-1}C^jBH^{n-1-j}& C^n
\end{pmatrix},\]

Let $f+g$ be recurrent for $T.$ 

Then $H^n(f)\in F$ and
$\sum_{j=0}^{n-1}C^jBH^{n-1-j}(f)+C^n(g)\in G$ and therefore
\[\delta d(H^n(f),f)\leq d(T^n(f+g),f+g) \]
which implies that $f$ is recurrent for $H.$ 

Assume now that $\{f_k+g_k\}_{k=1}^{\infty}$ is dense in $E$ and each vector is recurrent for $T.$ Then $\{f_k\}_{k=1}^{\infty}$ is dense in $F$ and therefore $H$ is recurrent.
\end{proof}
\begin{thm}
  Let $A=(a_{i,j})\in{L(\omega)}$ be lower triangular. If $K=\mathbb{C}$, $A$ is recurrent if and only if  $A=BDB^{-1}$ where $D$ a diagonal matrix with unimodular entries and $B$ and its inverse are in $L(\omega).$
\end{thm}
\begin{proof}
    Let $A$ be a lower triangular matrix. 
    Assume that $A$ is recurrent. Then, according to the Lemma \ref{partition}, $A$ restricted to $span{\{e_j:j=1,\cdots n\}}$ is recurrent, and we have that the elements on the diagonal are unimodular, Theorem 4.1 in \cite{Costakis}.
   
    For every $k\in{\mathbb{N}}$ there is a vector $v_k\in{\mathbb{C}^{\mathbb{N}}}$ such that $supp{(v_k)}\subset{\{j\in{\mathbb{N}}:k\le{j}\}}$ and $Av_k=a_{k,k}v_k$; moreover, $v_k$ can be chosen such that its $k$th entry is equal to $1$.
    To prove the claim above, we need to show that for each given $k$ there are complex numbers $x_{i,k}$, $i\ge{k}$ such that 
    $$a_{k+1, k}+a_{k+1,k+1}x_{k+1,k}=a_{k,k}x_{k+1, k}$$
    $$a_{k+2, k}+a_{k+2,k+1}x_{k+2,k}+a_{k+2, k+2}x_{k+2, k}=a_{k,k}x_{k+2, k}$$
    \center{$\vdots$}
    $$a_{k+l, k}+a_{k+l,k+1}x_{k+l,k}+\ldots+a_{k+l, k+l}x_{k+l, k}=a_{k,k}x_{k+l, k}$$
    \begin{center}{$\vdots$}
    \end{center}
    \par
    If the diagonal entries $a_{i, i}$ are distinct, the system has a solution. If there are some repeated entries in the diagonal, using Lemma \ref{finite}, we can assign  $x_j=0$ whenever $a_{k,k}=a_{j,j}$ and $k<j.$
    \\
    Let $B$ be the matrix in $L(\omega)$ whose columns are $v_1, v_2, \ldots, v_k, \ldots$. Then $B$ is a lower triangular matrix whose diagonal entries are all equal to $1$ and therefore $B\in L(\omega)$. $B$ is invertible and $B^{-1}$ can be calculated easily and is also lower triangular. Thus, $B^{-1}AB$ is a a diagonal matrix whose diagonal entries are the diagonal entries of the original matrix $A$. We have shown that $A$ is similar to a diagonal operator with unimodular entries.
   \par
    Conversely, let $A$ be a lower triangular matrix similar to a diagonal matrix with unimodular entries $a_{k, k}$. We will assume that $A$ itself is such a diagonal matrix. The result is a consequence of the following: For a numerable subset $\{\lambda_j:j=1,2,\cdots\}$ of the unit circle, by applying Cantor's diagonal method, we can obtain a subsequence $n_p$ such that $\lim_{p\to \infty}\lambda_j^{n_p}=1 \text{  for all  }j.$
       \end{proof}

\section{An Interlude}

\vskip .1in
\par
In this section, we present a non-recurrent multiplication operator defined by a Cantor set of positive measure and we study 
images and preimages  of sets of positive Lebesgue measure on $\mathbb T$ by the function of  $F_n(z)=z^n$.
Let $\mathbb T$ be the unit circle, $\mathbb T \subset \mathbb C.$
The normalized Lebesgue measure of a set $E\subset \mathbb T$ \ will be denoted by $|E|.$
\begin{prop}\label{prop1}
Let $\mathbb T$ be the unit circle.  The mapping $F_n(z)=z^n$ where $z\in \mathbb T$ satisfies that for each measurable set $E$

\[(i)~| F_n^{-1}(E)|=|E|\]
and
\[(ii)~|E|\leq | F_n(E)|\]

\end{prop}

\begin{proof} (i) It is enough to prove it for open sets consisting of a union of finite intervals, 
\[E=\cup_{j=1}^p(e^{2\pi ia_j},e^{2\pi i b_j}) ~~\text{ with }~~ 0\leq a_1<b_1<\cdots<a_p<b_p\leq 1\]
whose (normalized) measure is $\Sigma_{j=1}^p(b_j-a_j).$ The inverse image is
 \[F_n^{-1}(E)=\cup_{k=0}^{n-1}\cup_{j=1}^p(e^{2\pi i (\frac{a_j+k}{n})},e^{2\pi i (\frac{b_j+k}{n})})\]
 which has the same measure. 
 
 (ii) Let again $E$ be a finite union of open intervals. $E$ can be written as
 \[E=\cup_{k=0}^{n-1}V_ke^{2\pi i (k/n)}\]
where 
\[V_k=\cup_{j=1}^{p_k}(e^{2\pi i\frac{ a_{j,k}}{n}},e^{2\pi i \frac{b_{j,k}}{n}})\]

$0\leq a_{1,k}<b_{1,k}<\cdots<a_{p,k}<b_{p,k}\leq 1$ and $0\leq k \leq n-1.$

Then \[|E|=\sum_{k=0}^{n-1}|V_k|.\]
 Also  if       
$ V=\cup_{k=0}^{n-1}V_k ,$  then    $|V_k|\leq|V|,$ for all k.

Since $F_n(E)=\{\xi^n:\xi \in V\},$ we have $|F_n(E)|=n|V| $ and therefore
\[|E|=\sum_{k=0}^{n-1}|V_k|\le{n|V|}= |F_n(E)|\]
\end{proof}

\begin{rem}
(i) above means that $F_n$ is invariant with respect to the Lebesgue measure, and (ii) that $F_n$ is measuring increasing. The following corollary is a consequence of the second part of this proposition.
\end{rem}

\begin{cor} Let $\{n_j\}_{j=1}^{\infty}$ be any subsequence of $\mathbb N$ and  $E$ a set of positive L-measure. Then $F_{n_j}(z)=z^{n_j}$ cannot go to 1 a.e on $E$ with respect to L-measure.
\end{cor}
\begin{proof} Let $0<c<|E|$ and let $V$ be an open set centered at $1\in \mathbb T$ such that $|V|=c.$ Then 
\[|F_{n_j}(E)\setminus V|\geq|E|-c,\]
and any $\xi \in F_{n_j}(E)\setminus V$ satisfies that $\sin(c/2)|\leq |\xi-1|.$ \end{proof}

In the next proposition we improve the above corollary by showing that $F_n(E)$ is almost dense in the unit circle.

Recall that for a sequence of sets $A_n$ the $\overline {lim }$ is defined as
\[\limsup _{k\to\infty}A_n=\cap_{k=1}^{\infty}\cup_{j=k}^{\infty}A_j.\]

\begin{prop}
 Let $\{n_j\}_{j=1}^{\infty}$ be any subsequence of $\mathbb N$ and  $E$ a set of positive L-measure. Then 
\[|\limsup _{j\to\infty} F_{n_j}(E)|=1.\]    
\end{prop}

\begin{proof}
Since $|E|>0$ almost all its points are Lebesgue points. Let $\xi \in E$ be one of them. For $\epsilon>0$ let $\delta >0$ such that for $\xi$ centered in the open set $V$ holds that
\[\text{   if   }  |V|<\delta \text{   then   } \frac{|V\setminus (V\cap E)|}{|V|}<\epsilon.\]
Let $|V|=\frac{1}{n_j}<\delta.$ Then $F_{n_j}$restricted to $V$ is one-to-one and $F_{n_j}(V)$ is $\mathbb T$ except a point.
Moreover $|F_{n_j}(E\cap V)|>1-\epsilon$ which implies that for each open interval $W\subset \mathbb T$ there is a point $\nu \in F_{n_j}(E)\cap W.$ We also have that
\[|\limsup _{j\to\infty} F_{n_j}(E)|\geq 1-\epsilon.\]
Since $\epsilon$ could be arbitrarily small, the proposition is proved. \end{proof}

\begin{quest}
Let $\{n_j\}_{j=1}^{\infty}$ be any subsequence of $\mathbb N$ and $B\subset \mathbb T$ be a nondenumerable set. Does there exist $\xi \in B $  such that \[\overline{\{\xi^{n_j}:k\leq j\}}=\mathbb T ?\]
\end{quest}
If the $n_j$ is the full sequence the answer is yes since there exist a $\xi$ which is not a root of the unit, i.e, $\xi=e^{2\pi i s}$ with $s$ not a rational number. This is Kroenecker's theorem. Moreover, in this case Weyl proved that the sequence is equidistributed, p. 105 in \cite{Stein}.

\vskip .1in

{\bf Example.} This is related to Example 7.9 in \cite{Costakis}. It is well-known that there are Cantor type sets that have positive measure. To construct one it suffices to modify the length of the interval in the complement of the Cantor set. Our Cantor set has positive measure $1-c$ where $0<c<1;$ and the corresponding operator $M_{\phi}$ is not recurrent. We define the open and dense set $V$ in a similar fashion to the complement of the canonical Cantor set contained in $[0,1].$ We set

\[V=\cup_{k=1}^{\infty}\cup_{j=1}^{2^{k-1}}V_{j,k}  ~~\text{ and }~~F=[0,1]\setminus V=\cap_{k=1}^{\infty}\cup_{s=1}^{2^k}F_{s,k}\]
where $V_{j,k}$ is a middle third but now $|V_{j,k}|=c/3^k$ and the open intervals are enumerated from left to right. 
For each $k$,  $F_{s,k}$ are disjoint closed intervals with the same length and they are also enumerated from left to right. They are defined as
\[\cup_{s=1}^{2^k}F_{s,k}=[0,1]\setminus \cup_{t=1}^k \cup_{j=1}^{2^{t-1}}V_{j, t}.\] 
Observe that
\[F_{2j-1,k}\cup F_{2j,k}\subset F_{j,k-1}.\]

The function $\phi$ is defined on $[0,1]$ as $\phi(t)=e^{2\pi if(t)}$ where $f(t)$ on
$V$ is such that  
\[f|_{V_{1,1}}=1/2;~f|_{V_{1,2}}=1/4, ~f|_{V_{2,2}}=3/4; \cdots\] 
and so forth, and then $f$ is extended continuously to each of the points of the set $F$ defined above. As in the case of the devil's staircase, the function $f$ is one to one on $F$ except in the endpoints of the components of $V.$
Our next result is the key for showing the no recurrence of  
\[M_{\phi}: L^2([0,1],d\theta)\longrightarrow  L^2([0,1],d\theta).\] 

\begin{prop} \label{prop5} Let $T\subset F$ be a closed set of positive measure . Then
\[|T|\leq |f(T)|\leq \frac{1}{1-c}|T|.\]
\end{prop}
\begin{proof}
Since the length of $F_{s,k}$ for $s=1,\cdots,2^k$ is $|F_{1,k}|$ it follows that
\[2^k|F_{1,k}|=1-\sum_{s=1}^k\frac{c2^{s-1}}{3^s}=1-c\Big(1-\frac{2^k}{3^k}\Big).\]

Thus $|F_{1,k}|\geq 2^{-k}(1-c)$ and since $|f(F_{1,k})|=2^{-k}$ we have that
\[|F_{1,k}|\leq |f(F_{1,k})|\leq \frac{1}{1-c}|F_{1,k}|.\]

For each $k$, let's define 
\[N_k=\{s:F_{s,k}\cap T\neq \emptyset\} ~~\text{  and  }~~ T_k=\cup_{s\in N_k}F_{s,k}.\]
Then
\[ T=\cap_{k=1}^{\infty}T_k~~\text{  and  }~~f(T)=\cap_{k=1}^{\infty}~f(T_k)\]
and since $T_k\subset T_{k-1}$ the measure of $T$ and $f(T)$ satisfy
\[|T|=\lim_{k\to \infty}|T_k|~~\text{  and  }~~|f(T)|=\lim_{k\to \infty}|f(T_k)|\]
which proves the proposition. \end{proof}
Since $M_{\phi}^nh(t)=e^{n2\pi if(t)}h(t),$  Proposition \ref{prop1} (ii) and Proposition \ref{prop5} imply that $(\phi)^n(t)=e^{n2\pi if(t))}$ doesn't go to 1 almost everywhere. Thus
Theorem 7.6 of \cite{Costakis} says that $M_{\phi}$ is not recurrent.

    \section{Composition Operators on Function Spaces}
   
    Let $\mathbb{D}$ denote the open unit disk in the complex plane. For each sequence of positive numbers $\beta=\{\beta_n\}_n$ the weighted Hardy space $H^2(\beta)$ is the Hilbert space of functions analytic on $\mathbb{D}$ for which the norm induced by the inner product
    $$\big{<}\sum_{n=0}^\infty{a_nz^n}, \sum_{n=0}^\infty{b_nz^n}\big{>}=\sum_{n=0}^\infty{a_n\overline{b_n}\beta_n^2}$$
    is finite. The monomials form a complete orthogonal system and so they are dense in $H^2(\beta).$ Also, convergence in $H^2(\beta)$ implies uniform convergence on compact subsets of the unit disk. While we will study operators on $H^2(\beta)$ in general, we will focus on weighted Dirichlet spaces $S_\nu$, which are the weighted Hardy spaces with weights $\beta_n=(n+1)^\nu$, where $\nu$ is a real number. If $\nu_1>\nu_2$, then $S_{\nu_1}$ is strictly contained in $S_{\nu_2}$; if $\nu>\frac{1}{2}$, then $S_\nu$ is contained in the disk algebra $\mathcal{A}$.
    In the following, we will be interested in the composition operators induced by linear fractional maps (LFT). If $\phi$ is a linear fractional map that maps the open unit disk into itself, then the composition operator $C_\phi$, defined by $C_\phi(f)=f\circ{\phi}$, is bounded on all the $S_\nu$ spaces (Proposition 1.3 in \cite{Gallardo-Montes}).
    
    \begin{rem}
    Let $\alpha$ and $\nu$ be real numbers with $\alpha >\nu$. Since $S_\alpha\subset{S_\nu}$, $S_\alpha$ is dense in $S_\nu$ (polynomials are dense in both) and $||f||_\nu\le{||f||_\alpha}$, by the Comparison Principle (Proposition 2.1 in \cite{Karim}) we have that every continuous operator on $S_\nu$ that maps $S_\alpha$ to itself and which is recurrent on $S_\alpha$ is also recurrent on $S_\nu$. In particular, every composition operator induced by a linear fractional self-map of $\mathbb{D}$ that is recurrent on $S_\alpha$ will also be recurrent on $S_\nu$, whenever $\nu<\alpha$. Also, if such $T$ is not recurrent on $S_\nu$ implies that $T$ is not recurrent on $S_\alpha$ with $\alpha>\nu$.

    The following result provides additional information about the  relation between different spaces $S_{\nu}.$ 
    \end{rem}
    Let $X$ and $Y$ be Banach spaces with $X \subset Y.$ Recall that we say that $X$ is compactly embedded in $Y,$ denoted by $X \subset \subset Y$ if

(i) $$||x||_Y \leq C ||x||_X $$ for some constant $C$ and all $x\in X.$
\\
and
\\
(ii) Every bounded sequence in $X$ is precompact in $Y.$

The following proposition should be folklore.

\begin{prop} Let $X$ and $Y$ be Banach spaces with $X$  compactly embedded in $Y.$
If $Y$ is infinite dimensional, then $X$ is a first category set in $Y.$
\end{prop}

\begin{proof}
 The closure of the unit ball of $X$ with the norm $|| \cdot||_Y$  is compact and also is nowhere dense; otherwise, the unit ball of $Y$ would be compact, which is not. 
\end{proof}

\begin{prop} Let $\alpha=\{\alpha_k\}_k$ and $\beta=\{\beta_k\}_k$ be two sequences of positive real numbers such that
 \[0<\alpha_k \leq \beta_k \text{ for } k=0,1,2,\cdots \text{ and } 
\lim_{k\to \infty} \frac{\alpha_k}{\beta_k}=0.\]
Then \[\mathcal H^2(\beta) \subset \subset\mathcal H^2(\alpha).\]
\end{prop}

\begin{proof} Let $f_m$ in $\mathcal H^2(\beta)$ be such that $||f_m||_{\mathcal H^2(\beta)}\leq 1$ for all $m,$
and \[f_m(z)=\sum_{k=0}^{\infty}c_{k,m}z^k.\]
Let $c_m=\{c_{k,m}\}_k.$ Then for each $m$
\[c_m\in X_{k=0}^{\infty}\{w:|w|\leq\frac{1}{\beta_k}\}.\] 
Tihonov's theorem says that the last set, endowed with the product topology, is compact. By relabeling if necessary we may assume that $\{c_m\}_m$ is a Cauchy sequence in the product space. We prove now that  $\{f_m\}_m$  is also a Cauchy sequence in $\mathcal H^2(\alpha).$
Let $\epsilon >0.$  Let $N_0$ be such that $N_0<k$ implies that 
\[\frac{\alpha_k}{\beta_k}<\frac{\epsilon}{2\sqrt 2}\]
There is also an $N_1$ such that if $N_1\leq n,m$ then
\[|c_{k,m}-c_{k,n}|<\frac{\epsilon}{\sqrt {2\sum_{k=0}^{N_0}\alpha_k^2}} \text{ for } 0\leq k\leq N_0.\]
For  $N_1\leq n,m$
    \[ ||f_m-f_n||^2_{\mathcal H^2(\alpha)}=\]\[\sum_{k=0}^{N_0}|c_{k,m}-c_{k,n}|^2\alpha_k^2
  +  \sum_{k=N_0+1}^{\infty}|c_{k,m}-c_{k,n}|^2\alpha_k^2 \leq\]
    
    \[\frac{\epsilon^2}{2} +\frac{\epsilon^2}{8}||f_m-f_n||^2_{H^2(\beta)}\leq \epsilon^2\]
    This shows that  $\{f_m\}_m$ is also a Cauchy sequence in $\mathcal H^2(\alpha),$ and therefore the proof is complete \end{proof}

\begin{rem}
If $\nu<\mu$ then $S_{\mu} \subset \subset S_{\nu}$ and  $S_{\mu}$ is a first category subset of $S_{\nu}.$
\end{rem} 

The following theorem summarizes the results in Section 2 of the paper by Karim, Benchihev and Amouch \cite{Karim}, which are also illustrated in their Table 3. One of the ingredients they use is the Comparison Principle for recurrent operators, Proposition 2.1 in \cite{Karim}, which is modeled by the same principle for hypercyclic and supercyclic operators. They also use the Denjoy-Wolff theorem, which is a standard and very important tool for dealing with analytic self maps in the unit disk. 
\begin{thm} \label{lft}
Let $\phi$ be a linear fractional map of $D$ with $\phi(D)\subset{D}$ and let $C_\phi$ be the composition operator with symbol $\phi$ on the space $S_\nu$, $\nu\in{\mathbb{R}}$.
\begin{enumerate}
    \item If $\phi$ is not an elliptic automorphism, then $C_\phi$ is recurrent if and only if it is hypercyclic.
    \item If $\phi$ is an elliptic automorphism, then $C_\phi$ is recurrent on $S_\nu$ for all values of $\nu$. 
\end{enumerate}
 
\end{thm}
\begin{rem}
We can add that in the elliptic automorphism case, the operator $C_\phi$ is rigid on every space $S_\nu.$ Moreover, $C_\phi$ is uniformly rigid if and only if $\phi$ is conjugate to a rational rotation.
\end{rem}

\begin{rem}
     The cases of the parabolic non-automorphisms and of the hyperbolic-parabolic automorphisms  have not been treated correctly in \cite{Karim}, the issue being the bad behavior of functions in these spaces at the boundary, as discussed below. The proofs we present for these cases circumvent this issue. 
     
     In the first case, Theorem 2.8 \cite{Karim} the authors assumed that the recurrent vector they considered in the proof is a function that is continuous at the boundary fixed point, but this is not necessarily the case. It is known that there are functions in the Bergman space $S_{-\frac{1}{2}}$ that do not have radial limits almost everywhere, and, in fact, there is a function in this space that fails to have a limit at every point of the unit circle. 
     
     \par
We have been unable to locate a reference for the statement above, so in lieu of that, we have the following proposition.

 \begin{prop} In each $S_{-\nu}$ with $0<\nu $ there is an $f$ such that for any $e^{i\theta}$ we have that
$$\lim_{r\to 1}f(re^{i\theta}) \hspace{.1in} \text{does not exist}.$$
\end{prop}
\begin{proof}
    The function which satisfies the requirements is of the form 
    $$f(z)=\sum_{j=1}^{\infty}(m_j+1)^{\nu/2}z^{m_j}$$
    where the sequence $\{m_j\}$ goes to $\infty$ very rapidly. We will choose that sequence as a subsequence of a sequence $\{n_j\}$ which is also rapidly increasing and satisfies
    $$\sum_{j=1}^{\infty}\frac{1}{(n_j+1)^{\nu}}<\infty,$$
    which will ensure that the function $f(z)$ lies in $S_{-\nu}$.
    Set $m_1=n_1$ and consider the annulus
    $$A_1=\{r:\frac{1}{2} \leq r^{m_1}\leq \frac{3}{4}\},$$ which implies that for $r\in A_1$ we have  $(1/2)^{\frac{1}{m_1}}\le{r}\le{(3/4)^{\frac{1}{m_1}}}.$ Let $0<\epsilon<\frac{1}{4}$ 
    and therefore $\sum_{j+1}^{\infty}\epsilon^j<\frac{1}{3}.$
    
    Choose $m_2$ as one $n_k$ with $2\leq k$ such that 
   
    $$(3/4)^{\frac{m_2}{m_1}}(1+m_2)^{\nu/2}<\frac{\epsilon}{4}(1+m_1)^{\nu/2}.$$
    
    This condition we imposed on $m_2$ ensures that $(1+m_1)^{\nu/2}|z|^{m_1}>(1+m_2)^{\nu/2}|z|^{m_2}$ for $z\in{A_1}$.
    The annulus $A_2$ is defined as 
    \[A_2=\{r:\frac{1}{2} \leq r^{m_2}\leq \frac{3}{4}\}.\]
We will define $A_p$ as \[A_p=\{r:\frac{1}{2} \leq r^{m_p}\leq \frac{3}{4}\},\] and we will define $m_p$ for all $p\in{\mathbb{N}}$ such that  
 the term $(m_p+1)^{\nu/2}z^{m_p}$ will be dominant in $f(z)$ when $|z|=r\in A_p.$

To this end, assume that we have chosen $m_1,\cdots,m_p.$ We choose $m_{p+1}$ such that the following two conditions are satisfied for
$1\leq j\leq p$
\begin{eqnarray}
\label{eq:1}(3/4)^{\frac{m_{p+1}}{m_j}}(1+m_{p+1})^{\nu/2}<\frac{\epsilon^{p+1-j}}{4}(1+m_j)^{\nu/2},\end{eqnarray} which will ensure that for $z \in A_j$ the term containing $z^{m_j}$ dominates the term containing $z^{m_{p+1}}$ and so all the following terms. (The ones containing $z^{m_k}$ with $j<k.$)
\begin{eqnarray}
\label{eq:2}\sum_{j=1}^p(1+m_j)^{\nu/2}<\frac{1}{4}(1+m_{p+1})^{\nu/2}, \end{eqnarray} which  will ensure that for $z \in A_{p+1}$ the term containing $z^{m_{p+1}}$ dominates the previous terms. (The ones containing $z^{m_j}$ with $j<p+1.$) 

For $|z|\in A_1$ we have that
\[|f(z)|\geq \frac{1}{2}(1+m_1)^{\nu/2}-\sum_{p=1}^{\infty}(3/4)^{\frac{m_{p+1}}{m_1}}(1+m_{p+1})^{\nu/2}\geq\]
    \[\frac{1}{2}(1+m_1)^{\nu/2}-\sum_{p=1}^{\infty}\frac{\epsilon^p}{4}(1+m_1)^{\nu/2}>\frac{5}{12}(1+m_1)^{\nu/2}\]
For 
\[|z|\in A_{p+1}=\{r:\frac{1}{2} \leq r^{m_{p+1}}\leq \frac{3}{4}\}\]
we have that $$|f(z)|\ge{(1+m_{p+1})^{\nu/2}|z|^{m_{p+1}}-\sum_{j=1}^p{(1+m_j)^{\nu/2}{|z|^{m_j}}}-\sum_{s=2}^p{(1+m_{p+s})^{\nu/2}{|z|^{m_{p+s}}}}}\ge{}$$ 
$$1/2(1+m_{p+1})^{\nu/2}-\sum_{j=1}^p{(1+m_j)^{\nu/2}}-\sum_{s=2}^p{(1+m_{p+s})^{\nu/2}{(3/4)^{\frac{m_{p+s}}{m_{p+1}}}}},$$
where for $r^{m_{p+1}}$ we used in the first term the bound below of $1/2$ and in the third term the bound above of $3/4$, whereas in the middle term we used that $r<1.$ 
By using (\ref{eq:1}) to control the third term and (\ref{eq:2}) to control the middle term, we have that the last expression is bounded below by 

\[\frac{1}{2}(1+m_{p+1})^{\nu/2}-\frac{1}{4}(1+m_{p+1})^{\nu/2}-\sum_{s=2}^{\infty}\frac{\epsilon^{p+s-p-1}}{4}(1+m_{p+1})^{\nu/2}\geq\]
\[\frac{1}{6}(1+m_{p+1})^{\nu/2}.\]
As $p$ goes to infinity, the radii of the annulus $A_p$ converge to $1$, while the lower bound for $|f(z)|$ on $A_p$ goes to $\infty$. This shows that $\displaystyle{\limsup_{r\to 1}|f(re^{i\theta})|=\infty}$ for every $e^{i\theta}\in{\mathbb{T}}$, which proves the proposition. \end{proof}

\begin{rem}
Given that the polynomials are dense in $S_{-\nu}$, the set
\[E=\{h+\frac{1}{n}z^kf: h \in \mathcal{A}, f \text{ like in the proposition and } n,k \in \mathbb N\}\]
is dense in $S_{-\nu}$ , and every function $g$ in it satisfies that $\lim_{r \to 1}g(r\lambda)$ does not exist for $\lambda \in \mathbb T.$ Moreover, let

\[N_{Le}=\{g\in S_{-\nu}:\lim_{r \to 1}g(r\lambda)\text{ does not exist for } \lambda \in \mathbb T\},\]
and
\[N_{Lae}=\{g\in S_{-\nu}:\lim_{r \to 1}g(r\lambda)\text{ does not exist for almost all } \lambda \in \mathbb T\}.\]
Then
\[\{fg:f\in N_{Le}\text{ and }g\in \mathcal{A}, g\neq 0\} \subset N_{Le},\]
and
\[\{fg:f\in N_{Lae}\text{ and }g\in H^{\infty}, g\neq 0\} \subset N_{Lae},\]
and since the Hardy space $S_0$ is dense in $S_{-\nu}$ so is 
$N_{Lae}+S_0.$

It can also be proved that $N_{Le}$ is a residual set \cite{B-S}
\end{rem}

     In the second case, Theorems 2.4 and 2.5 \cite{Karim}, although the recurrent $f$ in the Dirichlet space converges non-tangentially at almost every point in the unit circle, it might not converge at the chosen point 1. It might be possible to fix this problem but rather than doing that, we give a different proof. 
\end{rem}    

    \begin{thm} Let $\phi$ be  a parabolic or a hyperbolic automorphism, then $C_{\phi}$ is recurrent on $S_{\nu}$ if and only if $\nu <\frac{1}{2}.$ 
    \end{thm}
\begin{proof}
Since for $v<\frac{1}{2}$, $C_\phi$ is hypercyclic, it follows that in this case $C_\phi$ is recurrent. Let $\nu=\frac{1}{2}$, when $S_\nu=\mathcal{D},$ the classical Dirichlet space.

Claim: No composition operator is recurrent on $D.$ 

\noindent Then, by using the Comparison Principle, these operators cannot be recurrent on $S_\nu$ for $\nu > \frac{1}{2}$. 

Therefore, the theorem will be proved if we prove the claim for $\mathcal {D}.$ Without loss of generality, we can assume that the attractive fixed point of $\phi$ is 1.
    Let $\mathcal{D}_0=\{f \in \mathcal{D}: f(0)=0\}$. Thus,
\[ \mathcal D= \mathcal D_0\oplus C,\]
where the sum is the orthogonal sum and $C$ is the one-dimensional space of constants. Every function $f$ in $\mathcal{D}$ can be uniquely expressed as $f(z) =  \tilde{f}(z)+f(0)$, with $\tilde{f}(0) = 0$.

Since a composition operator acts like the identity on the constants, we can express $C_\phi$ with respect to $\mathcal D= \mathcal D_0\oplus C$  as
\[C_\phi=\begin{pmatrix}
 \tilde{C_\phi}& 0\\
B & I
\end{pmatrix},\]
where $\tilde{C_\phi}(\tilde f)(z)=f(\phi(z)) -f(\phi(0))$
and $B(\tilde f)(z)=f(\phi(0))-f(0).$

Lemma 3.11 says that if $C_\phi$ is recurrent on $\mathcal{D}$, then its compression $\tilde{C}_\phi$ is recurrent on $\mathcal{D}_0$. (In this application of the Lemma, any $\delta\le{\frac{1}{\sqrt{2}}}$ works.)  Moreover, if a vector $f$ is recurrent for $C_\phi$ on $\mathcal{D}$, then the vector $\tilde{f}=f-f(0)$  is recurrent for $\tilde{C_\phi}.$

We are now ready to complete the proof of the theorem.
    Assume that $C_\phi$ is recurrent on $\mathcal{D}$. Then $\tilde{C}_\phi$ is recurrent on $\mathcal{D}_0$. From Proposition 3.2 in \cite{Gallardo-Montes}, we know that the adjoint $(\tilde{C}_\phi)^*=\tilde{C}_{\phi^{-1}}$ is a unitary operator on $\mathcal{D}_0$. By Proposition 3.1 in \cite{Costakis}, a recurrent operator that is power bounded has the property that every vector in its domain is a recurrent vector. This means that every vector in $\mathcal{D}_0$ is a recurrent vector for $\tilde{C}_\phi$. In particular, $f(z)=\tilde{f}(z)=z$ is recurrent for $\tilde{C}_\phi$, that is, there exists an increasing sequence of natural numbers $(n_k)_k$ such that 
    
    \begin{equation} \label{rec}
||\tilde{\phi}_{n_k}(z)-z||^2_{\mathcal{D}_0}\rightarrow{0} \hspace{.05in} \text{as} \hspace{.05in} k\to\infty.
\end{equation}
We now show that $(\phi_{n_k})_k$ converges to $g(z)=z+1$ in $\mathcal{D}$ as $k\to\infty$. Indeed,

    \[||\phi_{n_k}(z)-g(z)||^2_{\mathcal{D}}=|\phi_{n_k}(0)-1|^2+\int_\mathbb{D}|(\phi'_{n_k}(z)-1)|^2dA(z).\]

By the Denjoy-Wolff theorem the first term in the sum above converges to $0$ as $k\to\infty$, while the second term converges to $0$ according to \ref{rec}. 

But convergence in $\mathcal{D}$ implies pointwise convergence, so let $z_0$ be a point in the open unit disk that is different from $0$. We then have $\phi_{n_k}(z_0)\rightarrow{z_0+1}$, which contradicts the conclusion of the Denjoy-Wolff theorem, which states that $\phi_{n_k}(z_0)\rightarrow{1}$. (An additional way of seeing this is that each $\phi_{n_k}$ is  a self map of the disc while its limit $g(z)=z+1$ is not a self map.) This contradiction shows that $C_\phi$ is not recurrent on $\mathcal{D}$.
\end{proof}

\begin{thm} [Parabolic Non-automorphism symbol] \label{parabolicnonauto}

A parabolic non automporphism $C_{\varphi}$ is never recurrent for $S_{\nu}$

    \end{thm}
 Theorem 6.13 in \cite{Costakis} states that in this case $C_\varphi$ is not recurrent on $H^2=S_0$, so according to the Comparison Principle, it is not recurrent on $S_\nu$ for all $\nu\ge{0}$. Therefore, we need to consider only when $\nu<0.$
    
Given that recurrence is preserved under similarity, we can assume that \[\varphi_n(z)=\frac{(2-na)z+na}{-naz+2+na},\] where $a=x+iy$ and $y\in \mathbb R$ and $x=Re(a)>0.$ 

We use the notation on page 49 \cite{Gallardo-Montes}.  Thus 
\[\varphi_n=\overline{\gamma_n}+\overline{\alpha_n}K_{w_n}\]
where 
\[\overline{\gamma_n}=\frac{na-2}{na},~~~~\overline{\alpha_n}=\frac{4}{na(na+2)}, ~~~~\overline{w_n}=\frac{na}{na+2}\] and $K_w(z)=(1-\overline {w}z)^{-1}.$ Thus\[\varphi_n(0)=\overline{\gamma_n}+\overline{\alpha_n}=\frac{na}{na+2}.\]
The successive derivatives of $\varphi_n $ are, for $1\leq k$ \[\varphi^{(k)}_n(z)=k!~\overline{\alpha_n}(\overline{w_n})^k(1-\overline{w}z))^{-k-1},\]
and when evaluated at zero
\begin{equation}\label{zeros}\varphi^{(k)}_n(0)=k!\frac{4}{(na+2)^2}\left(\frac{na}{na+2}\right)^{k-1}=k!\varphi^{(1)}_n(0)~\overline{w_n}^{k-1}.\end{equation}
\begin{rem}
 Since   $S_{\nu}\subset Hol(\mathbb D)$ and  $f\in S_{\nu} $ implies that
 $f^{(k-1)}\in S_{\nu-k+1}$, by applying Proposition 2.16 in \cite{Gallardo-Montes}, we have the following estimate where $C$ does not depend on $z$
 \[\left|\frac{df^{(k-1)}}{dz}(z)\right|=|f^{(k)}(z)|\leq C\frac{1}{(1-|z|^2)^{\frac{1+2k-2\nu}{2}}}\]
 whenever $\nu-k+1<\frac{1}{2}$ since $3-2(\nu-(k-1)=1+2k-2\nu.$
\end{rem}
\begin{lem} \label{lem}
For each $f\in S_{\nu}$ and $k=1,2,3,\cdots$
\[\lim_{n\to \infty}|f^{(k)}(\varphi_n(0))||\varphi_n^{(1)}(0)|^{k}=0\] 
whenever $\frac{1-2k}{2}<\nu<0.$
\end{lem}
\begin{proof}
Let $a=x+iy$ with $Re (a)=x>0.$ Then $|2+na|^2=(nx+2)^2+n^2y^2$ and $|2+na|^2-n^2|a|^2=4(nx+1)$. 

By using the values of  $\varphi_n(0), ~~\varphi_n^{(1)}(0)$ already obtained and the last remark
\[|f^{(k)}(\varphi_n(0))||\varphi_n^{(1)}(0)|^{k}\leq \frac{C}{(1-|\frac{na}{na+2}|^2)^{\frac{1+2k-2\nu}{2}}}\frac{4^k}{|na+2|^{2k}}=\]
\[4^kC\frac{|na+2|^{1-2\nu}}{(|na+2|^2-|na|^2)^{\frac{1+2k-2\nu}{2}}}\leq
     4^kC\frac{((nx+2)^2+n^2y^2)^{\frac{1-2\nu}{2}}}{(4(nx+1))^{\frac{1+2k-2\nu}{2}}}\leq M (nx)^{\frac{1-2\nu-2k}{2}}\]
where the last inequality holds if $M$ is large enough. The right hand-side goes to 0 when $n$ goes to infinity, and this completes the proof of the lemma.     
\end{proof}

We now have the tools for proving Theorem \ref{parabolicnonauto}.

\begin{proof}
The plan is the following: 

(i) Use the first derivative to see that the conclusion is true for
$-\frac{1}{2}<\nu.$ 

(ii) Use the second derivative to see that the conclusion is true for
$-\frac{3}{2}<\nu.$

(iii) Use the k-th derivative  to see that the conclusion is true for
$-\frac{2k-1}{2}<\nu.$

Let $f=\sum_{j=0}^{\infty}b_jz^j \in S_{\nu}$ be a recurrent vector for $C_{\varphi}.$

To prove (i) let $-\frac{1}{2}<\nu.$ Then, using Lemma \ref{lem} for $k=1$
\[b_1=\lim_{n\to \infty} (f\circ \varphi_n)^{(1)}(0)=0.\] 
This shows that the orthogonal projection of the set of the recurrent vectors on the space generated by $\{z\}$ is zero.

To prove (ii)
let $-\frac{3}{2}<\nu.$ 
We will show that $C_{\varphi}$ is not recurrent because

Claim: {\it The orthogonal projection on the space generated by $\{z,z^2\}$ of the recurrent vectors of $S_{\nu}$ is at most one dimensional.}

By recurrence of $f$ we have that
\[\lim_{n\to \infty}(f\circ \varphi_n)^{(1)}(0)=b_1
\text{  and   }\lim_{n\to \infty}(f\circ \varphi_n)^{(2)}(0)=2b_2\]
But
\[(f\circ \varphi_n)^{(2)}(z)=f^{(2)}(\varphi_n(z))(\varphi_n^{(1)}(z))^2+f^{(1)}(\varphi_n(z))\varphi_n^{(2)}(z)\]
Taking the limit for $n\to \infty$ and evaluating in 0 and using Lemma 2.3 and remembering that $\varphi_n^{(2)}(0)=\varphi_n^{(1)}(0)\overline{w_n}$ and $\overline{w_n} \to 1,$ we have that $2b_2=b_1$ proving the claim.

To prove (iii) let $\frac{1-2k}{2}<\nu.$
 We will show that $C_{\varphi}$ is not recurrent because 

 Claim: {\it The orthogonal projection of the
recurrent vectors on the space generated by $\{z,z^2,\cdots,z^k\}$ is at most $k-1$ dimensional.}

We will use the
Fa\`a di Bruno formula for the $m$th derivative of the composite of two functions, see, for instance, [Wolfram MathWorld]. In this case $f(\varphi_n(z)),$  the evaluations are for $z=0,$ and we use \ref{zeros} to put all the derivatives of $\varphi_n^{(s)}(0)$ in terms of $\varphi_n^{(1)}(0)$ and powers of $\overline {w_n}.$
\[(f\circ \varphi_n)^{(m)}(0)=\sum_{s=0}^m f^{(s)}(\varphi_n(0))(\varphi_n^{(1)}(0))^s\frac{m!}{j_1!\cdots j_m!}
\prod_{u=1}^m\left(\frac{1}{u!}\right)^{j_u}\overline {w_n}^{m-s}\]
where   $\sum_{u=1}^mj_u=s$ and $\sum_{u=1}^muj_u=m.$

On the other hand the recurrence of $f$ implies that
\[m!b_m=\lim_{n\to \infty} (f\circ \varphi_n)^{(m)}(0).\]

For $1\leq s \leq m$ we have that $\overline {w_n}^{m-s} \to 1$ when $n \to \infty.$ Since we are assuming that $f$ is recurrent, by induction we have that \[\lim_{n\to \infty} f^{(s)}(\varphi_n(0))(\varphi_n^{(1)}(0))^s=\text{linear combination of }b_1,\cdots, b_s\]
for any $1\leq s \leq m.$

When $m=k$  Lemma \ref{lem} says that 
\[\lim_{n\to \infty} f^{(k)}(\varphi_n(0))(\varphi^{(1)}_n(0))^{k}=0\]
which implies that
\[b_k=\text{linear combination of }b_1,\cdots, b_{k-1}.\]
This proves the claim and therefore the proof of the theorem is completed. \end{proof}

   To conclude, we discuss the recurrence of composition operators on the disk algebra $\mathcal{A}$ and on the algebra $H^{\infty}.$ While no composition operator on $\mathcal{A}$ can be hypercyclic  since such operators have bounded orbits (see \cite{Bulancea} for a discussion of the cyclic behavior of composition operators on $\mathcal{A}$ induced by LFTs), and we cannot talk about cyclicity of operators on $H^\infty$ because  $H^\infty$ is not separable, these operators could still be recurrent on these spaces. The result below shows that no operator that is not an elliptic automorphism is recurrent. It is worth emphasizing that the result no longer assumes an LFT symbol. In the following discussion, $A$ will denote either one of these two spaces.  
   The result we present relies on the characterization of surjective isometries for uniform algebras that appears in Hoffman \cite{Hoffman} on page 147.
   \begin{thm} \cite{Hoffman}
   Let $X$ be a compact Hausdorff space and let $A$ be a complex linear subalgebra of $C(X)$, the algebra of continuous complex-valued functions on $X$. Assume that $A$ contains the constant function $1$. Suppose $T$ is a one-one map of $A$ onto $A$ which is isometric:
   $$||Tf||_\infty=||f||_\infty.$$
   Then $T$ has the form 
   $$Tf=\alpha\phi{f}$$
   where $\alpha$ is a fixed function in $A$ which is of modulus $1$, $\frac{1}{\alpha}$ is in $A$, and $\phi$ is an algebra automorphism. In particular, if $T(1)=1$, then $T$ is multiplicative.
   \end{thm}
   When we apply this theorem to the algebra $A$ which is $H^\infty$, or $\mathcal{A}$, we obtain that the surjective isometry $T$ has the form 
   $$(Tf)(\lambda)=\alpha{f(\tau(\lambda))},$$
   for $f\in{A}$, where $\alpha$ is a constant of modulus $1$ and $\tau$ is a conformal of the open disk onto itself, i.e. $\displaystyle{\tau(z)=e^{i\theta}\frac{z-a}{1-\bar{a}z}}$, with $\theta\in{[0, 2\pi)}$ and $a\in\mathbb{D}$. Notice that $\tau$ is continous on $\overline{\mathbb{D}}$.
   \par We will also need Proposition 3.2 in \cite{Costakis} which describes properties of power bounded recurrent operators. 
    \begin{thm} No composition operator on $A$ induced by a non-constant self-map of $\mathbb{D}$ which is not an elliptic automorphism is recurrent on $A$.
    \end{thm}
    \begin{proof}
    Let $C_\phi$ be a composition operator on $A$ induced by the analytic self-map $\phi$ of $\mathbb{D}$. Given that the identity function on $\mathbb{D}$ is in $A$, $\phi$ is an element of ${A}$. The operator $C_\phi$ is a contraction on $A$, $||C_\phi(f)||\le{||f||}$ for all $f\in{A}$. Assume that $C_\phi$ is recurrent. By part (i) of Proposition 3.2 in \cite{Costakis}, it follows that $C_\phi$ is a surjective isometry. Therefore, by the theorem in \cite{Hoffman} cited above, we have that  $(C_\phi{f})(z)=\alpha{f(\tau(z))}$, where $\alpha$ is a constant of modulus $1$ and $\tau$ is a conformal automorphism of $\mathbb{D}$ onto itself, from which, using Proposition 2.3 in \cite{Costakis}, $C_\tau$ is recurrent. Also, by part (v) of Proposition 3.2 in \cite{Costakis}, Rec$(C_\tau)=A$. In particular, $\tau$ is a recurrent vector for $C_\tau.$ According to the Denjoy-Wolff theorem, there is a point $p$ in the closed unit disk such that $\tau_n\rightarrow{p}$ uniformly on compact sets, which contradicts the fact that $\tau$ is a recurrent vector for $C_\tau.$ Therefore, $C_\phi$ cannot be recurrent on $A$.
    \end{proof}
    \begin{rem} This approach could potentially be used to study the recurrence of endomorphisms of uniform algebras, which are composition operators whose symbol is a self-map of the maximal ideal space of the respective algebra.
    \end{rem}


\begin{thebibliography}{9}
\bibitem [1] {Bayart}
Bayart, Frédéric; Matheron, Étienne. Dynamics of linear operators. Cambridge Tracts in Math., 179 Cambridge University Press, Cambridge, 2009. xiv+337 pp. ISBN:978-0-521-51496-5
\bibitem [2] {Bes-Conejero}
Bès, Juan; Conejero, José A. Hypercyclic subspaces in omega. J. Math. Anal. Appl.316(2006), no.1, 16–23.
\bibitem [3] {Bulancea} Bulancea, Gabriela. Cyclic behaviour of certain operators on the disc algebra.
Complex Var. Elliptic Equ.67(2022), no.12, 3023–3035.
\bibitem [4] {B-S} Bulancea, G., Salas, H. Boundary behavior of analytic functions on some Banach spaces.
\bibitem [5] {Costakis} \label{b2}
Costakis, G., Manoussos, A. \& Parissis, I. Recurrent Linear Operators. Complex Anal. Oper. Theory 8, 1601–1643 (2014). https://doi.org/10.1007/s11785-013-0348-9
\bibitem [6] {Parissis} \label{b3}
Costakis, G., Parissis, I. Szemerédi's theorem, frequent hypercyclicity and multiple recurrence. Math. Scand.110(2012), no.2, 251–272.
\bibitem [7] {Gallardo-Montes} \label{b5}
Gallardo-Gutiérrez, Eva A., Montes-Rodríguez, A. The role of the spectrum in the cyclic behavior of composition operators. Mem. Amer. Math. Soc. 167 (2004), no. 791, x+81
\bibitem[8]{Hoffman} 
Hoffman, Kenneth. Banach spaces of analytic functions.
Reprint of the 1962 original Dover Publications, Inc., New York, 1988. viii+216 pp. ISBN:0-486-65785-X
\bibitem[9]{Karim}\label{b1}
   Karim, N., Benchiheb, O., Amouch, M. Recurrence of multiples of composition operators on weighted Dirichlet spaces. Adv. Oper. Theory 7 (2022), no. 2, Paper No. 23, 15 pp 
\bibitem[10]{Hector_omega} \label{b4}
Salas, H\'{e}ctor N. Eigenvalues and hypercyclicity in omega. Rev. R. Acad. Cienc. Exactas F\'{\i}s. Nat. Ser. A Mat. RACSAM, 105, no. 2 (2011): 379--388.
\bibitem [11] {Shapiro}
Shapiro, Joel H. Composition operators and classical function theory. Universitext Tracts Math. Springer-Verlag, New York, 1993. xvi+223 pp. ISBN:0-387-94067-7
\bibitem [12]{Stein}
Stein, Elias M.; Shakarchi, Rami. Fourier Analysis: An Introduction (Princeton Lectures in Analysis, Volume 1) Princeton University Press; Illustrated edition (April 6, 2003) ISBN-13 : 978-0691113845
\bibitem [13] {Stout}
Stout, Edgar Lee. The theory of uniform algebras. Bogden \& Quigley, Inc., Publishers, Tarrytown-on-Hudson, N.Y., 1971. x+509 pp.
\end{thebibliography}
\end{document}